\documentclass[11pt]{amsart}
\usepackage{amssymb,latexsym,amsmath,amsthm, mathrsfs, tikz, hyperref}
\usepackage[utf8]{inputenc}
\usepackage[T1]{fontenc}

 \newtheorem{theorem}{Theorem}[section]
 \newtheorem{cor}[theorem]{Corollary}
 
 \newtheorem{prop}[theorem]{Proposition}
 \theoremstyle{definition}
 
 \theoremstyle{remark}
 \numberwithin{equation}{section}

\newcommand{\dist}{{\mathop{\rm dist}}}
\newcommand{\tr}{{\mathop{\rm tr}}}
\newcommand{\rank}{{\mathop{\rm rank}}}

\newcommand{\A}{{\mathcal A}}
\newcommand{\B}{{\mathcal B}}

\newcommand{\hc}{{\mathcal H}}
\newcommand{\hk}{{\mathcal K}}
\newcommand{\sL}{{\mathcal L}}

\newcommand{\eps}{\varepsilon}
\newcommand{\norm}[1]{\left\lVert#1\right\rVert}

\newcommand{\bF}{{\mathbb{F}}}
\newcommand{\bR}{{\mathbb{R}}}
\newcommand{\bC}{{\mathbb{C}}}

\newcommand{\bra}{\langle}
\newcommand{\ket}{\rangle}

\newcommand{\tens}[1]{\mathbin{\mathop{\otimes}\limits_{#1}}}
\begin{document}


\title{Birkhoff-James orthogonality and applications : A survey}


\author[Grover]{Priyanka Grover}

\address{Department of Mathematics, Shiv Nadar University,\\
 Dadri,  U.P. 201314\\
India.}

\email{priyanka.grover@snu.edu.in}

\author[Sushil]{Sushil Singla}
\address{Department of Mathematics, Shiv Nadar University,\\
 Dadri,  U.P. 201314\\
India.}
\email{ss774@snu.edu.in}




\begin{abstract} 

In the last few decades, the concept of Birkhoff-James orthogonality has been used in several applications. In this survey article,  the results known on the necessary and sufficient conditions for Birkhoff-James orthogonality  in certain Banach spaces  are mentioned. Their applications in studying the geometry of normed spaces are given. The connections between this concept of orthogonality, and the Gateaux derivative and the subdifferential set of the norm function are provided. Several interesting distance formulas can be obtained using the characterizations of Birkhoff-James orthogonality, which are also mentioned. In the end, some new results are obtained.  

\end{abstract}

\maketitle



\section{Introduction}

Let $(V, \norm{\cdot})$ be a normed space over the field $\bR$ or $\bC$. For normed spaces $V_1, V_2$, let $\mathscr B(V_1,V_2)$ denotes the space of bounded linear operators from $V_1$ to $V_2$ endowed with the \emph{operator norm}, and let $\mathscr B(V)$ denotes $\mathscr B(V,V)$. Let $K(V_1, V_2)$ denotes the space of compact operators from $V_1$ to $V_2$. Let $\hc$ be a Hilbert space over $ \bR$ or $\bC$.  If the underlying field is $ \bC$, the inner product on $\hc$ is taken to be linear in the first coordinate and conjugate linear in the second coordinate. The notations $M_n(\bR)$ and $M_n(\bC)$ stand for $n\times n$ real and complex matrices, respectively.

Normed spaces provide a natural setting for studying geometry in the context of vector spaces. While inner product spaces capture the concept of the measure of an angle,  orthogonality of two vectors can be described without knowing the notion of measure of angle. For example, a vector $v$ is orthogonal to another vector $u$ in $\bR^n$ if and only if there exists a rigid motion $T$ fixing the origin such that the union of rays $\overrightarrow{0u}, \overrightarrow{0T(v)}, \overrightarrow{0T(u)}$ minus the open ray $\overrightarrow{0v}$ is the one dimensional subspace generated by $u$ \footnote{We learnt this characterization of orthogonality in $\bR^n$  from Amber Habib.}. This description of orthogonality by just using the notion of distance in $\bR^n$ motivates to try and define orthogonality in normed spaces. In this approach,  one can use the intuition about orthogonality in $\bR^n$ to guess the results in general normed spaces and then prove them algebraically. This has been done in  \cite{Amir, Birkhoff, James 1, Roberts}.

One of the definitions for orthogonality in a normed space suggested by Roberts \cite{Roberts}, known as \emph{Roberts orthogonality}, is defined as follows: elements $u$ and $v$ are said to be (Roberts) orthogonal if $\|v+t u\|=\|v-tu\|$ for all scalars $t$. In \cite[Example 2.1]{James 1}, it was shown that this definition  has a disadvantage that there exist normed  spaces in which two elements are Roberts orthogonal  implies that one of the element has to be zero.  In \cite{James 1}, two more inequivalent definitions of orthogonality in normed spaces were introduced. One of them is \emph{isosceles orthogonality} which says that $v$ is isosceles orthogonal to $u$ if $\|v+u\|=\|v-u\|$. The other one is called \emph{Pythagorean orthogonality}, that is, $v$ is Pythagorean orthogonal to $u$ if $\|v\|^2+\|u\|^2=\|v-u\|^2$. Note that if $V$ is an inner product space, all the above mentioned definitions are equivalent to the usual orthogonality in an inner product space. Isosceles and Pythagorean orthogonalities have geometric intuitions for the corresponding definitions.  In $\bR^n$, two vectors are isosceles perpendicular if and only if their sum and difference can be sides of an isosceles triangle, and  two vectors are Pythagorean perpendicular if there is a right triangle having the two vectors as legs. In \cite{James 1}, it was also proved that if $u$ and $v$ are two elements of a normed space, then there exist scalars $a$ and $ b$ such that $v$ is isosceles orthogonal to $av+u$ (see \cite[Theorem 4.4]{James 1}) and $v$ is Pythagorean orthogonal to $bv+u$ (see \cite[Theorem 5.1]{James 1}). So these definitions don't have the above mentioned weakness of Roberts orthogonality. 

In an inner product space, the following properties are satisfied by orthogonality. Let $u, u_1, u_2, v, v_1, v_2\in V$.

\begin{enumerate}
\item \textit{Symmetry:} If $v\bot u$, then $u\bot v$.
\item \textit{Homogeneity:} If $v\bot u$, then $av\bot bu$ for all scalars $a$ and $b$.
\item \textit{Right additivity:} If $v\bot u_1$ and $v\bot u_2$, then $v\bot (u_1+u_2)$. \\
\textit{Left additivity:} If $v_1\bot u$ and $v_2\bot u$, then $(v_1+v_2)\bot u$.
\item There exists a scalar $a$ such that $v\bot av+u$. (In $\bR^n$, this corresponds to saying that any plane containing a vector $v$ contains a vector perpendicular to $v$.)
\end{enumerate}

It is a natural question to study the above properties for any given definition of orthogonality. 
All the above definitions clearly satisfy symmetry. 
James \cite[Theorem 4.7, Theorem 4.8, Theorem 5.2, Theorem 5.3]{James 1} proved that if isosceles or Phythagorean orthogonality satisfy homogeneity or (left or right) additivity, then $V$  has to be an inner product space. These orthogonalities have been extensively studied in \cite{Amir,James 1,Roberts}.

In \cite{Birkhoff}, Birkhoff defined a concept of orthogonality, of which several properties were studied by James in \cite{James 2}.
An element $v$ is said to be \emph{Birkhoff-James} orthogonal to $u$ if $\norm{v}\leq\norm{v+ku}$ for all scalars $k$. 
The analogy in $\bR^n$ is that if two lines $L_1$ and $L_2$ intersect at $p$, then $L_1 \perp L_2$ if and only if the distance from a point of $L_2$ to a given point $q$ of $L_1$ is never less than the distance from $p$ to $q$.
This definition clearly satisfies the homogeneity property. In \cite[Corollary 2.2]{James 2}, it was shown that this definition also satisfies $(4)$.  But it lacks symmetry, for example, in $(\bR^2, \|\cdot\|_{\max})$, where $\|(t_1,t_2)\|_{\max}=\max\{|t_1|, |t_2|\}$, take  $v=(1, 1)$ and $ u=(1, 0)$ ($v\perp u$ but $u\not\perp v$). It is  not right additive, for example,  in $(\bR^2, \|\cdot\|_{\max})$,  take $ v=(1,1)$, $u_1=(1,0)$ and $u_2=(0,1)$. It is  also not left additive, for example, in $(\bR^2, \|\cdot\|_{\max})$,  take $ v_1=(1,1)$, $v_2=(0,-1)$ and $u=(1,0)$.

Let $W$ be a subspace of $V$. Then an element $v\in V$ is said to be Birkhoff-James orthogonal to $W$ if $v$ is Birkhoff-James orthogonal to $w$ for all $w\in W$. 
A closely related concept  is that of a \emph{best approximation to a point in a subspace}. An element $w_0\in W$ is said to be a best approximation to $v$ in $W$ if $\norm{v-w_0} \leq \norm{v-w}$ for all $w\in W$.
Note that $w_0$ is a best approximation to $v$ in $W$ if and only if $v-w_0$ is Birkhoff-James orthogonal to $W$. These are also equivalent to saying that $\dist(v, W):=\inf\{\norm{v-w}:w\in W\} = \norm{v-w_0}$. So $v$ is Birkhoff-James orthogonal to  $W$  if and only if $\dist(v, W)$ is attained at $0$. Therefore the study of these concepts go hand in hand (see \cite{Singer}).  
This is one of the reasons that this definition of orthogonality, even though not symmetric, is still being extensively studied in literature. Henceforward, orthogonality will stand for Birkhoff-James orthogonality. 

Recently, a lot of work has been done in the form of applications of this concept of orthogonality and the main goal of this survey article is to bring all the related work under one roof. In Section \ref{Sec2}, we mention the connections between  orthogonality and geometry of normed spaces. We also deal with the question as to when the   orthogonality is symmetric or (left or right) additive. This leads us to the study of various related notions like characterizations of smooth points and extreme points,  subdifferential set, $\varphi$-Gateaux derivatives etc. In Section \ref{Sec3}, characterizations of orthogonality in various Banach spaces are discussed along with some applications. In Section \ref{Sec4}, these characterizations are used to obtain distance formulas in some Banach spaces. Some of the stated results are new and  will appear in more detail in \cite{2019}. Theorems \ref{7}, \ref{7777} and \ref{21} are the new results given with proofs only here.

\section{Orthogonality and geometry of normed spaces}\label{Sec2}

A \emph{hyperplane} is a closed subspace of codimension one. A connection between the concept of orthogonality and hyperplanes is given in the next theorem. An element $v$ is orthogonal to a subspace $W$ if and only if there exists a linear functional $f$ on $V$ such that $\norm{f} =1$, $f(w) = 0$ for all $w\in W$ and $f(v) = \norm{v}$ (see \cite[Theorem 1.1, Ch. I]{Singer}). This is equivalent to the following. 

\begin{theorem}\label{th2.1}(\cite[Theorem 2.1]{James 2}): Let $W$ be a subspace of $V$. Let $v\in V$. Then $v$ is orthogonal to $W$ if and only if there is a hyperplane $H$ with $v$ orthogonal to $H$ and $W\subseteq H$.
\end{theorem}

By the Hahn-Banach theorem and Theorem \ref{th2.1}, it is easy to see that any element of a normed  space is orthogonal to some hyperplane (see \cite[Theorem 2.2]{James 2}). The relation between orthogonality and hyperplanes is much deeper. We first recall some definitions. For $v\in V$, we say $S\subseteq V$ \emph{supports} the closed ball $D[v, r]:=\{x\in V: \|x-v\|\leq r\}$ if $\dist(S, D[v,r]) = 0$ and $S\ \cap\ \text{Int } D[v,r] = \emptyset$.  This is also equivalent to saying that $\dist(v, S) = r$ (see \cite[Lemma 1.3, Ch. I]{Singer}). Let $v_0$ be an element of the boundary of $D[v,r]$. A hyperplane $H$ is called a \emph{support hyperplane to $D[v,r]$ at $v_0$} if $H$ passes through $v_0$ and supports $D[v,r]$, and it is called a \emph{tangent hyperplane to $D[v,r]$ at $v_0$} if $H$ is the only support hyperplane to $D[v,r]$ at $v_0$. A real hyperplane is a hyperplane in $V$, when $V$ is considered as a real normed space.

\begin{theorem} (\cite[Theorem 1.2, Ch. I]{Singer}) Let $W$ be a subspace of $V$. Let $v\in V$. Then $v$ is orthogonal to $W$ if and only if there exists a support hyperplane to $D[v,r]$ at $0$ passing through $W$ if and only if there exists a real hyperplane which supports the closed ball $D[v, \norm{v}]$ at $0$ and passes through $W$. 
\end{theorem}

A direct consequence follows. If  $W$ is a non-trivial subspace of $V$, then $0$ is the unique best approximation of $v$ in $W$ if and only if there exists a tangent hyperplane to $D[v,r]$ at $0$ passing through $W$ (see \cite[Corollary 1.5, Ch. I]{Singer}).

The above results are related to the questions as to when the orthogonality is (left or right) additive or symmetric. It was shown in \cite[Theorem 5.1]{James 2} that orthogonality is right additive in  $V$ if and only if for any unit vector $v\in V$, there is a tangent hyperplane to $D[v, \norm{v}]$ at $0$. There are other interesting characterizations for (left or right) additivity of  orthogonality. To state them, some more definitions are required.
A normed space $V$ is called a \emph{strictly convex} space if given any $v_1,v_2\in V$, whenever $\norm{v_1} +\norm{v_2} = \norm{v_1+v_2}$ and $v_2\neq 0$, then there exists a scalar $k$ such that $v_1=k v_2$. This is also equivalent to saying that if $\norm{v_1} = \norm{v_2} = 1$ and $v_1\neq v_2$, then $\norm{v_1+v_2} <2$.
The norm $\norm{\cdot}$  is said to be \emph{Gateaux differentiable} at $v$ if $$\lim\limits_{h\rightarrow 0} \dfrac{\norm{v+hu} - \norm{v}}{h}$$ exists for all $u\in V$.

Now we have the following characterizations for the orthogonality to be right additive in $V$.

\begin{theorem}\label{1} The following statements are equivalent.
\begin{enumerate}
\item Orthogonality is right additive.
\item Norm is Gateaux differentiable at each nonzero point.
\item For $v\in V$, there exists a unique functional $f$ of norm one on $V$ such that $f(v) = \norm{v}$.
\item For $v\in V$, there is a tangent hyperplane to $D[v, \norm{v}]$ at $0$.
\end{enumerate}
If  $V$ is a reflexive space, then the above are also equivalent to the  following statements.
\begin{enumerate}
\item[(5)] Any bounded linear functional on a given subspace of $V$ has a unique norm preserving Hahn-Banach extension on $V$.
\item[(6)] The dual space $V^*$ is strictly convex.
\end{enumerate}
\end{theorem}

\begin{proof}  Equivalence of $(1)$ and $(2)$ is proved in \cite[Theorem 4.2]{James 2} and equivalence of $(1)$, $(3)$ and $(4)$ is proved in  \cite[Theorem 5.1]{James 2}. For a  reflexive space, equivalence of $(1)$ and $(5)$ is given in  \cite[Theorem 5.7]{James 2}. Equivalence of $(5)$ and $(6)$  is a routine exercise in functional analysis.
\end{proof}

Characterization of inner product spaces of dimension three or more can be given in terms of (left or right) additivity or symmetry of orthogonality. Birkhoff \cite{Birkhoff} gave a necessary and sufficient condition for a normed space of dimension at least three to be an inner product space, and examples to justify the restriction on the dimension. James \cite[Theorem 6.1]{James 2} showed that a normed space of dimension at least three is an inner product space  if and only if orthogonality is right additive and symmetric if and only if the normed space is strictly convex and orthogonality is symmetric. Later, James improved his result and proved a much stronger theorem.
\begin{theorem}(\cite[Theorem 1, Theorem 2]{James 3}) Let $V$ be a normed space of dimension at least three. Then $V$ is an inner product space if and only if orthogonality is symmetric or left additive.
\end{theorem}
A characterization of orthogonality to be symmetric or left additive in a normed space of dimension two can be found in \cite{Alonso}.  Several other necessary and sufficient conditions for a normed space to be an inner product space are given in \cite{Alonso, James 1}. This problem has also been extensively studied in \cite{Amir, Singer}.

An element $v$ is called a \emph{smooth point} of $D[0,\|v\|]$  if there exists a hyperplane tangent to $D[v, \norm{v}]$ at $0$. We say $v$ is a smooth point if it is a smooth point of $D[0,\|v\|]$. Equivalently, $v$ is a smooth point if there exists a unique affine hyperplane passing through $v$ which supports $D[0, \norm{v}]$ at $v$ (such an affine hyperplane is called the affine hyperplane tangent to $D[0, \norm{v}]$ at $v$). A normed space is called \emph{smooth} if all its vectors are smooth points. By Theorem \ref{1}, we get that orthogonality in a normed space is right additive if and only if the normed space is smooth.  We also have that $v$ is a smooth point if and only if the norm function is Gateaux differentiable at $v$:

\begin{theorem}\label{th2.5} Let $v\in V$. The norm function is Gateaux differentiable at $v$ if and only if there is a unique $f\in V^*$ such that $\|f\|=1$ and $f(v)=\|v\|$. In this case, the Gateaux derivative of the norm at $v$ is given by $\text{Re }f(u)$ for all  $u\in V$. In addition, for $u\in V$, we have that $v$ is orthogonal to $u$ if and only if $f(u) = 0$.
\end{theorem}

Smooth points and this connection with Gateaux differentiability was studied in \cite{Abatzoglou, Diestel,  Keckic, Keckic 1} and many interesting results can be obtained as their applications.  Let the space of continuous functions on a compact Hausdorff space $X$ be denoted by $C(X)$ and let the space of bounded continuous functions on a normal space $\Omega$ be denoted by $C_b(\Omega)$. Ke$\check{\text c}$ki$\acute{\text c}$ \cite[Corollary 2.2, Corollary 3.2]{Keckic 1} gave characterizations of smooth points in $C(X)$ and $C_b(\Omega)$. A characterization of smooth points in $\mathscr B(\hc)$ was given in \cite[Corollary 3.3]{Keckic}. For $\hc$ separable, Abatzoglou \cite[Corollary 3.1]{Abatzoglou} showed that the operators in $\mathscr B(\hc)$ of unit norm which are also smooth points  are dense in the unit sphere of $\mathscr B(\hc)$. In $K(\hc)$, this result was first proved by Holub \cite[Corollary 3.4]{Holub}. Heinrich \cite[Corollary 2.3]{Heinrich} generalized this result for $K(V_1, V_2)$, where $V_1$ is a separable reflexive Banach space and $V_2$ is any normed space. He proved that the operators which attain their norm at a unique unit vector (upto scalar multiplication) are dense in $K(V_1, V_2)$.

In this paragraph, $\hc$ is a separable Hilbert space. Schatten \cite{Schatten} proved that $D[0,1]$ in $K(\hc)$  has no extreme points. In \cite{Holub}, the geometry of $K(\hc)$ and its dual $\mathscr B_1(\hc)$, the \emph{trace class}, was studied by characterizing the smooth points and extreme points of their closed unit balls. It was shown in \cite[Corollary 3.1]{Holub} that the trace class operators of rank one and unit norm  are exactly the extreme points of $D[0,1]$ in  $\mathscr B_1(\hc)$. The space $\mathscr B_1(\hc)$ is predual of $\mathscr B(\hc)$ and hence is isometrically isomorphic to a subspace of $\mathscr B(\hc)^*$. An interesting result in \cite[Corollary 3.3]{Abatzoglou} is that all the trace class operators of rank one and unit norm  are also extreme points of $D[0,1]$ in $\mathscr B(\hc)^* $. In \cite{Heinrich}, this study was continued to understand the geometry of  $K(V_1, V_2)$, $\mathscr B(V_1, V_2)$ and the weak tensor product of $V_1$ and $V_2$, where  $V_1$ and $V_2$ are Banach spaces. Characterizations of Gateaux differentiability and Fr\'echet differentiability of the norm at an operator  $T$ in these spaces were obtained. For Schatten classes of $\hc$, this problem was addressed in \cite[Theorem 2.2, Theorem 2.3]{Abatzoglou}. In \cite[Theorem 3.1]{Abatzoglou}, another characterization of Fr\'echet differentiability of the norm at $T$ in $\mathscr B(\hc)$  was given, an alternative proof  of which can be found in \cite[Theorem 4.6]{Sain 1}. In \cite[Corollary 2.2]{Heinrich}, a necessary and sufficient condition for $0\neq T\in K(V_1, V_2)$ to be a smooth point is obtained, where $V_1$ is a reflexive Banach space and $V_2$  is any Banach space. It is shown that such a $T$ is a smooth point if and only if $T$ attains its norm on the unique unit vector $x_0$ (up to a scalar factor) and $T x_0$ is a smooth point. (This was proved for $K(\hc)$ in \cite[Theorem 3.3]{Holub}.)  Recently, as an application of orthogonality, it was shown in  \cite[Theorem 4.1, Theorem 4.2]{Sain 1} that this characterization also holds when $V_2$ is any normed space (not necessarily complete). 

{If $T\in \mathscr B(V_1, V_2)$ attains its norm on the unique unit vector $x_0$ (up to a scalar factor) and $T x_0$ is a smooth point of $V_2$, then $T$ is said to satisfy \emph{Holub's condition} (see \cite{Hennefeld}). Then Theorem 4.1 and Theorem 4.2 in \cite{Sain 1} say that for a reflexive Banach space $V_1$ and any normed space $V_2$, smooth points of $K(V_1, V_2)$ are exactly those operators which satisfy Holub's condition. This characterization may not hold if $T$ is not compact (see \cite[Example (a)]{Hennefeld}) or when $V_1$ is not a reflexive space (see \cite[Example (b), Example (c)]{Hennefeld}). In the case when $V_1$ is not a reflexive space, usually some extra condition is needed along with Holub's condition to characterize smooth points. For example, Corollary 1 in \cite{Yonuis 1} states that for $1<p,r<\infty$, a necessary and sufficient condition for $T\in\mathscr B(l^p, l^r)$ to be a smooth point is that $T$ satisfies Holub's condition and $\dist(T, K(l^p, l^r))<\|T\|$. As an application of orthogonality, it is proved in \cite[Theorem 4.5]{Approx 4} that for any normed spaces $V_1, V_2$, if $T\in\mathscr B(V_1, V_2)$ attains its norm and is a smooth point, then $T$ satisfies Holub's condition and $\dist(T, K(V_1, V_2))<\|T\|$. The converse is true when $V_1$ is a reflexive Banach space and $V_2$ is any Banach space and $K(V_1, V_2)$ is an $M$-ideal in  $\mathscr B(V_1, V_2)$ (see \cite[Theorem 4.6]{Approx 4}). It is an open question whether or not these extra assumptions on $V_1$ and $V_2$ are required. Some sufficient conditions,  along with Holub's condition,  for an operator to be smooth are also known when the underlying field is $\bR$. If $V_1$ is a real Banach space and $V_2$ is a real normed space, one such condition for smooth points in $\mathscr B(V_1, V_2)$ is given in \cite[Theorem 4.3]{Sain 1}. When $V_1$ and $V_2$ are any real normed spaces, such  conditions are given in \cite[Theorem 3.2]{Sain 2} and \cite[Theorem 3.4]{new}. The extra condition which along with Holub's condition gives the characterization for smoothness of any non zero norm attaining operator $T\in\mathscr B(V_1, V_2)$  (for any real normed spaces $V_1, V_2$) is obtained in \cite[Theorem 3.3]{new}.   For further study of smooth points, we refer the readers to \cite{Yonuis 2, Kittaneh,  Rao 1, Rao 2, Rao 3, werner}.}

Extreme points of $D[0,1]$ are important because of Krein-Milman theorem. Along with the extreme points, the faces of $D[0,1]$ in any normed space have also been of interest.  (Note that the extreme points are exactly faces with a single element.) Let $M_n(\bR)$ or $M_n(\bC)$ be equipped with any \emph{unitarily invariant norm}, $|||\cdot|||$ (that is, for any matrix $A$  and $U,U'$ unitary, $|||UAU'|||=|||A|||$). Then there is a unique \emph{symmetric gauge function} $\Phi$ on $\bR^n$ such that $|||A|||=\Phi((s_1(A), \ldots, s_n(A))$, where $s_i(A)$ are singular values of $A$ arranged as $s_1(A)\geq \cdots\geq s_n(A)$. Zi\k{e}tak \cite[Theorem 5.1]{Zietak 1} showed that a  necessary and sufficient condition for a matrix $A$ to be an  extreme point of the closed unit ball in $(M_n(\bR), |||\cdot|||)$ is that $(s_1(A),\ldots,s_n(A))$ is an extreme point of the closed unit ball in $(\bR^n, \Phi)$. This result was extended to $M_n(\bC)$ in \cite[Theorem 1]{So} (these results also follow from the results in \cite{Arazy}). Li and Schneider  \cite[Proposition 4.1]{Li Schneider} characterized the extreme points of $D[0,1]$ in $M_n(\bR)$ and $M_n(\bC)$, equipped with the dual of an induced norm. 
In $\mathscr B(\hc)$, the extreme points of $D[0,1]$ are exactly the isometries and the coisometries (see \cite[p. 263]{Halmos}). It was proved in \cite[Theorem 2.5]{extreme} that $A\in\mathscr B(\hc)$ is an isometry or a coisometry if and only if $\|A\| = 1$ and  $A$ is \emph{right symmetric} (for definition, see \cite{Sain 5}). So the extreme points of $D[0,1]$ in $\mathscr B(\hc)$ are precisely those operators which are of unit norm and are also right symmetric.
There is also a concept of a \emph{left symmetric operator}, the study of which can be found in  \cite{Sain 5, Symmetric 1, Symmetric, Wrong 1}.

Theorem 2, Theorem 3 and Theorem 4 in \cite{So} give characterizations of  proper closed faces in $M_n(\bC)$, equipped with Schatten p-norms. {Theorem 4.1 in \cite{Zietak} and     the discussion above it give a characterization of faces of $D[0,1]$ in $(M_n(\bC),|||\cdot|||)$ as follows: $F$ is a face of $(M_n(\bC),|||\cdot|||)$ if and only if there exists $A\in M_n(\bC)$  such that  $F$ is a face of $\partial |||A|||^*$, the \emph{subdifferential set} of $|||\cdot|||^*$ at $A$, where $|||\cdot|||^*$ is the dual norm of $|||\cdot|||$.
In a normed space $V$, the {subdifferential} set of a continuous convex function $g:V\rightarrow \bR$ at $v\in V$ is denoted by $\partial g(v)$, and  is defined as the set of bounded linear functionals $f \in V^*$ satisfying the below condition: 
\begin{equation*}g(u)-g(v)\geq \text{Re } f(u-v) \quad \text{for all } u\in V.\label{eq 1.3.1}\end{equation*} It is a non-empty weak* compact convex subset of $V^*$. The below two propositions are easy to check. We refer the readers to \cite{thesis, hiriart} for more details. 
\begin{prop}\label{p1}
Let $v\in V$. Then
\begin{equation*}
\partial \|v\|=\{f\in V^*: \text{Re }f(v)=\|v\|, \|f\|\leq1\}.\label{eq 1.3.12}
\end{equation*}
\end{prop}


In particular, for  $A\in M_n(\bC)$, 
$$\partial |||A|||=\{G\in M_n(\bC): \text{Re } \tr(G^* A)=|||A|||, |||G|||^*\leq 1\}.$$
\begin{prop}\label{p2} Let $u,v\in V$. Then we have  $$\lim\limits_{t\rightarrow 0^+} \dfrac{\norm{v+tu} - \norm{v}}{t} = \max\{\text{Re }f(u) : f\in V^*, \norm{f} = 1, f(v) = \norm{v}\}.$$\end{prop} 
Using this, Watson \cite[Theorem 4]{Watson} gave a characterization of $\partial |||\cdot|||$ in the space $(M_n(\bR), |||\cdot|||)$. Zi\k{e}tak \cite[Theorem 3.1, Theorem 3.2]{Zietak} improved this result and showed the following.

\begin{theorem}(\cite[Theorem 3.1, Theorem 3.2]{Zietak})  For $A\in M_n(\bC)$,
\begin{align*}\partial |||A||| &= \{ U \text{diag}(d_1,\ldots,d_n) U'^*: A=U \Sigma U'^* \text{ is a singular value}\\
& \text{ decomposition } of A,
 \sum s_i(A) d_i=|||A|||=\Phi((s_1,\ldots, s_n)),\\ &\ \Phi^*((d_1,\ldots, d_n))=1\}.\end{align*}\
\end{theorem}  
In \cite[Theorem 1]{watson93}, the above result was proved using a different approach. 
For the operator norm $\|\cdot\|$ on $M_n(\bC)$, we have the following.
\begin{cor}(\cite[Example 3]{watson93})\label{cor 2.9} For $A\in M_n(\bC)$,
$$\partial \|A\|=\text{conv }\{uv^*:\|u\|=\|v\|=1, Av=\|A\|u\},$$ where $\text{conv } S$ denotes the convex hull of a set $S$.
\end{cor}
Along the similar lines of \cite{Watson} (that is, by using Proposition \ref{p2}), the subdifferential set of the \emph{Ky Fan $k$-norms}, $\|\cdot\|_{(k)}$, on $M_n(\bC)$ was obtained in \cite{2017}. In \cite{2013,2014,thesis, 2017}, the subdifferential set was used to obtain characterizations of orthogonality in $M_n(\bC)$, equipped with various norms.

Actually the right hand derivative has a deeper connection with orthogonality as explored by Ke$\check{\text c}$ki$\acute{\text c}$ \cite{Keckic 2}, where the author introduced the notion of \emph{$\varphi$-Gateaux derivatives}: for $u,v\in V$ and $\varphi\in [0,2\pi)$, the {$\varphi$-Gateaux derivative} of norm at $v$ in the direction $u$ is defined  as $$D_{\varphi, v}(u) = \lim\limits_{t\rightarrow 0^+} \dfrac{\norm{v+te^{\iota\varphi}u} - \norm{v}}{t}.$$
These always exist for any two vectors $u$ and $v$ (see \cite[Proposition 1.2]{Keckic 2}). A characterization of orthogonality follows.
\begin{theorem}\label{x}(\cite[Theorem 1.4]{Keckic 2}) Let $u,v\in V$. Then $v$ is orthogonal to $u$ if and only if $$\inf\limits_{0\leq\varphi \leq 2\pi} D_{\varphi, v}(u) \geq 0.$$\end{theorem}
In \cite[Theorem 2.4]{Keckic}, the expression for the $\varphi$-Gateaux derivative of the  norm  on $\mathscr B(\hc)$ was obtained. 
Using the above  proposition, a characterization of orthogonality in $\mathscr B(\hc)$ was given in \cite[Corollary 3.1]{Keckic}, which was first proved in \cite{Bhatia} using a completely different approach. This characterization of orthogonality  and many of its generalizations are the main content of the next section.

\section{Characterizations and applications of  orthogonality}\label{Sec3}

Bhatia and $\check{\text S}$emrl  \cite{Bhatia} gave characterizations of orthogonality in $\mathscr B(\hc)$ in terms of orthogonality of vectors in the underlying Hilbert space $\hc$. These are given in the next two theorems. An independent proof of Theorem \ref{3} was also given by Paul \cite{Radii}.

\begin{theorem}\label{2}(\cite[Theorem 1.1]{Bhatia}) Let $A, B\in M_n(\bC)$. Then $A$ is orthogonal to $B$ if and only if there exists a unit vector $x\in\bC^n$ such that $\norm{Ax} = \norm{A}$ and $\bra Ax | Bx\ket = 0$.\end{theorem}

Let $\norm{\cdot}$ be any norm on $\bC^n$ or $\bR^n$ and let $\|\cdot\|'$ be the corresponding
{induced norm}  on $M_n(\bC)$ or $M_n(\bR)$, respectively.
It was conjectured in \cite[Remark 3.3]{Bhatia} that  a matrix $A$ is orthogonal to another matrix $B$ in $(M_n(\bC), \norm{\cdot}')$ if and only if there exists a unit vector $x\in\bC^n$ such that $\norm{Ax} = \norm{A}'$ and $Ax$ is orthogonal to $Bx$ in $(\bC^n,\norm{\cdot})$. Li and Schneider \cite[Example 4.3]{Li Schneider}  gave an example  to show that the conjecture is false in $M_n(\bC)$ as well as in $M_n(\bR)$. In $(M_n(\bC),\|\cdot\|')$ (or $(M_n(\bR),\|\cdot\|')$), a matrix $A$ is said to satisfy \emph{B$\check{\text S}$ property} if for any matrix $B$, whenever $A$ is orthogonal to $B$, there exists a unit vector $x$ such that $\norm{Ax} = \norm{A}'$ and $Ax$ is orthogonal to $Bx$ in $(\bC^n, \|\cdot\|)$ (or $(\bR^n, \|\cdot\|)$) (see \cite[Definition 1.1]{Sain 3}). 
It was proved in \cite{Ben} that  $(\bR^n, \norm{\cdot})$ is an inner product space if and only if every $A\in M_n(\bR)$ satisfies  B$\check{\text S}$ property. In \cite[Theorem 2.2]{Sain 3}, it was shown that if $(\bR^n, \norm{\cdot})$ is a smooth space and $A\in M_n(\bR)$  is such that $\{x\in\bR^n : \norm{x} =1, \norm{Ax} = \norm{A}'\}$ is a countable set with more than two points, then $A$ does not satisfy  B$\check{\text S}$ property.  Example 4.3 in \cite{Li Schneider} for $M_n(\bR)$ is a special case of this.
It was shown in  \cite[Corollary 2.1.1]{Sain 4} that  if $A\in M_n(\bR)$ attains its norm at exactly two points, then $A$ satisfies B$\check{\text S}$ property. A generalization of this theorem can be found in \cite[Theorem 3.1]{last}. In \cite[Theorem 2.1]{Sain 4}, another sufficient condition for  $A$ to satisfy B$\check{\text S}$ property was given.
If $(\bR^n, \norm{\cdot})$ is a strictly convex space, then the collection of the matrices which satisfy B$\check{\text S}$ property are dense in $M_n(\bR)$ (see \cite[Theorem 2.6]{Sain 3}). 
\begin{theorem}\label{3}(\cite[Remark 3.1]{Bhatia}, \cite[Lemma 2]{Radii}) Let $\hc$ be a complex Hilbert space. Let $A, B\in \mathscr B(\hc)$. Then $A$ is orthogonal to $B$ if and only if there exists a sequence of unit vectors $h_n\in \hc$ such that $\norm{Ah_n} \rightarrow \norm{A}$ and $\bra Ah_n | Bh_n\ket \rightarrow 0$, as $n\rightarrow \infty$.\end{theorem}

When $\hc$ is an infinite dimensional space, one can't expect to get a single vector $h$ in Theorem \ref{3} such that $\norm{Ah} = \norm{A}$ and $\bra Ah | Bh\ket = 0$. In fact it was proved in \cite[Theorem 3.1]{Sain 1} that for $A\in\mathscr B(\hc)$, the following are equivalent. \\(a) For $B\in\mathscr B(\hc)$, $A$ is orthogonal to $B$ if and only if there exists a unit vector $h\in \hc$ such that $\norm{Ah} =\norm{A}$ and $\bra Ah | Bh\ket = 0$. \\(b) There exists  a finite dimensional subspace $\hc_0$ of $\hc$ such that $$\{h\in\hc : \norm{h} =1 , \norm{A} = \norm{Ah}\} = \{h\in \hc_0: \norm{h} = 1\} \text{ and }\norm{A|_{\hc_0^{\bot}}} < \norm{A}.$$

It was noted in \cite{Bhatia} that Theorem \ref{2}  is equivalent to saying that for $A,B\in M_n(\bC)$, \begin{equation}\label{95} \dist(A, \bC B) = \max\left\{\big|\bra Ax | y\ket\big| : \norm{x} = \norm{y} = 1 \text{ and } y\bot Bx\right\}.\end{equation} It is natural to expect that in the  infinite dimensional case, we should have for $A, B\in B(\hc)$,  \begin{equation}\label{96}\dist(A, \bC B) = \sup\left\{\big|\bra Ax | y\ket\big| : \norm{x} = \norm{y} = 1 \text{ and } y\bot Bx\right\}.\end{equation}  This was indeed shown to be true in \cite{Bhatia} by using the approach given in \cite[p. 207]{Two pages}. We would like to point out that the book \cite{Two pages} deals with only separable spaces. However the arguments can be modified by replacing the  sequence of finite rank operators converging pointwise to the identity operator by a net with this property. Since the same proof as in \cite[p. 207]{Two pages} was used in the proof of Theorem 2.4 of \cite{Wrong 1}, a similar modification is required there too.
 
Later, several authors have used different methods to prove Theorem \ref{3}. One of these techniques was given in \cite[Remark 2.2]{2012}  using  a  different distance formula  \cite[Proposition 2.1]{2012}. Another approach in \cite[Corollary 3.1]{Keckic} uses Theorem \ref{x} and the expression for the $\varphi$-Gateaux derivative of the norm on $\mathscr B(\hc)$ (which is given in \cite[Theorem 2.4]{Keckic}).  Using Theorem \ref{x}, W$\acute{\text o}$jcik \cite{Wojcik} extended Theorem \ref{2}  for compact operators between two reflexive Banach spaces over $\bC$:
\begin{theorem} (\cite[Theorem 3.1]{Wojcik}) Let $V_1$ and $ V_2$ be reflexive Banach spaces over $\bC$. Suppose $A, B\in K(V_1, V_2)$ and $A\neq 0$. Then $A$ is orthogonal to $B$ if and only if $$\min\{\max\{D_{\varphi, Ay}(By) : \norm{y} =1, \norm{Ay} = \norm{A}\} : \varphi\in [0, 2\pi)\}\geq 0 .$$\end{theorem}
Let $\hc_1$ and $\hc_2$ be Hilbert spaces. In $K(\hc_1, \hc_2)$, the above theorem reduces to saying that for $A,B\in K(\hc_1,\hc_2)$, $A$ is orthogonal to $B$ if and only if  there is a unit vector $h\in \hc_1$ such that $\norm{Ah} = \norm{A}$ and $\bra Ah | Bh\ket = 0$. But this is not always the case with  reflexive Banach spaces. 

An alternate proof of Theorem \ref{2}  was given in \cite{2013} 
by first giving a characterization of $\|A\| \leq \|A+tB\|$ for all $t\in \bR$ using Corollary \ref{cor 2.9}, and then extend the result to complex scalars to obtain Theorem \ref{2}. In \cite{Wrong 1}, $\|A\| \leq \|A+tB\|$ for all $t\in \bR$ is termed as \emph{$A$ is $r$-orthogonal to $B$}, and the same characterization as in \cite{2013} is given for $r$-orthogonality using a different approach. Using the same idea as in \cite[Theorem 2.7]{2013}, a proof of Theorem \ref{2} was given in \cite[Corollary 2.2]{Wrong 1}.
 
 The technique of using the subdifferential set as done in \cite{2013} has advantages that it gives a way to generalize Theorem \ref{2} to characterize orthogonality to a subspace of $M_n(\bC)$. 

\begin{theorem}\label{t3.4}(\cite[Theorem 1]{2014}) Let $A\in M_n(\bC)$. Let $m(A)$ denotes the multiplicity of maximum singular value $\norm{A}$ of A. Let $\B$ be any (real or complex) subspace of $M_n(\bC)$. Then $A$ is orthogonal to $\B$ if and only if there exists a density matrix $P$ of complex rank atmost $m(A)$ such that $A^*AP = \norm{A}^2P$ and $tr(APB^*) = 0$ for all $B\in\B$.\end{theorem}

Theorem \ref{t3.4} can be expressed in terms of \emph{states} on $M_n(\bC)$. Let $\A$ be a unital $C^*$-algebra over $\bF(=\bR$ or $\bC)$ with the identity element  $1_{\A}$. For $\bF=\bC$, a state on $\A$ is a linear functional $\phi$ on $\A$ which takes $1_{\A}$ to $1$ and positive elements of $\A$ to non-negative real numbers. For $\bF=\bR$, an additional requirement for $\phi$ to be a state is that $\phi(a^*) = \phi(a)$ for all $a\in\A$. Let $S_{\A}$ denotes the set of states on $\A$.
Recently,  the authors noticed  in \cite{2019} that  if $P$ is a density matrix such that $\tr(A^*AP) = \norm{A}^2$, then $P$ is a matrix of complex rank atmost $m(A)$ such that $A^*AP = \norm{A}^2P$ (the proof of this fact is along the lines of proof of Theorem 1.1 in \cite{Bhatia}).
Due to this fact, the above theorem can be restated in terms of states on $M_n(\bC)$ as follows: $A$ is orthogonal to $\B$ if and only if there exists   $\phi\in S_{M_n(\bC)}$ such that $\phi(A^*A) = \norm{A}^2$ and $\phi(AB^*) =0$ for all $B\in \B$. In a general complex $C^*$-algebra $\A$, it was shown in \cite[Theorem 2.7]{2012} that an element $a\in \A$ is orthogonal to another element $b\in \A$ if and only if there exists $\phi\in S_{\A}$ such that $\phi(a^*a)=\|a\|^2$ and $\phi(a^*b)=0$. A different proof of this result was given in \cite[Proposition 4.1]{2013}. Theorem 6.1 in \cite{rieffel 1} shows that if $\B$ is a unital $C^*$-subalgebra of a complex $C^*$-algebra $\A$ and if a Hermitian element $a$ of $\A$  is orthogonal to $\B$, then there exists $\phi\in S_{\A}$ such that $\phi(a^2)=\|a\|^2$ and $ \phi(ab+b^*a)=0$  for all $b\in\B$. Recently, the authors have extended these results to any (real or complex) $C^*$-algebra $\A$ for any element $a\in \A$ and any subspace $\B$ of $\A$ (see \cite{2019}). These are given in the next theorem.

 If $\A$ is a complex (or real) unital $C^*$-algebra, then the triple  $(\hc, \pi, \xi)$ denotes a cyclic representation of $\A$, where $\hc$ is a complex (or real) Hilbert space and $\pi: \A\rightarrow \mathscr B(\hc)$ is a $^*$-algebra map  such that $\pi(1_\A) = 1_{\mathscr B(\hc)}$ and $\{\pi(a)\xi: a\in\A\}$ is dense in $\mathscr B(\hc)$.  For $\phi\in S_{\A}$, there exists a cyclic representation $(\hc, \pi, \xi)$ such that $\phi(a) = \bra\pi(a)\xi | \xi\ket$ for all $a\in\A$ (see \cite[p. 250]{conway}, \cite[Proposition 15.2]{goodearl}). 
\begin{theorem}\label{4}([\cite[Corollary 1.3]{2019}) Let $a\in\A$. Let $\B$ be a subspace of $\A$. Then the following are equivalent.
\begin{enumerate}\item $a$ is orthogonal to $\B$. \\
\item There exists $\phi\in S_{\A}$ such that $\phi(a^*a) = \norm{a}^2$ and $\phi(a^*b) =0$ for all $b\in\B$. \\
\item There exists a cyclic representation $(\hc, \pi, \xi)$ such that $\norm{\pi(a)\xi} = \norm{a}$ and $\bra\pi(a)\xi | \pi(b)\xi\ket = 0$ for all $b\in\B$.\end{enumerate}\end{theorem}

When $\A=C(X)$,  Theorem \ref{4} and  Riesz Representation Theorem yield the following theorem by  Singer \cite[Theorem 1.3, Ch. I]{Singer}.
\begin{cor}(\cite[Theorem 1.3, Ch. I]{Singer}) Let $f\in C(X)$. Let $\B$ be a subspace of $C(X)$. Then $f$ is  orthogonal to $\B$ if and only if there exists a probability measure $\mu$ on $X$ such that $$\dist(a, \B)^2 = \int\limits_X |f|^2 d\mu \text{ and }\int\limits_X \overline{f}h d\mu = 0$$ for all $h\in\B$.\end{cor} The condition $$ \norm{f}_{\infty}^2 = \dist(a, \B)^2 = \int\limits_X |f|^2 d\mu$$ is equivalent to saying that the support of $\mu$ is contained in the set $\{x\in X: | f(x)| = \norm{f}_\infty\}$. When $\B$ is one dimensional, this was proved in  \cite[Corollary 2.1]{Keckic 1} using Theorem \ref{x}.

Characterizations of orthogonality have been studied in several normed spaces.
Using Theorem \ref{x},  a characterization of orthogonality in $C_b(\Omega)$ was obtained in \cite[Corollary 3.1]{Keckic 1}.  In the Banach spaces $L^1(X, \nu)$ and $c_0$, Theorem \ref{x} was used to obtain such characterizations in \cite[Example 1.6, Example 1.7]{Keckic 2}.  For a separable Hilbert space $\hc$, expressions for $\varphi$-Gateaux derivative of the norms  on $\mathscr B_1(\hc)$ and $K(\hc)$ were given in \cite[Theorem 2.1, Theorem 2.6]{Keckic 2}  and were used to give characterizations of orthogonality in these spaces in \cite[Corollary 2.5, Corollary 2.8]{Keckic 2}.  Using tools of subdifferential calculus, characterizations of orthogonality in $(M_n(\bC),\|\cdot\|_{(k)})$ are given in  \cite[Theorem 1.1, Theorem 1.2]{2017}. A  necessary condition for orthogonality of a matrix $A$ to a subspace in $(M_n(\bC),\|\cdot\|_{(k)})$ is given in \cite[Theorem 1.3]{2017}. Under the condition that $s_k(A)>0$, the same  condition is shown to be sufficient also. Using \cite[ Theorem 1.1, Ch. II]{Singer}, a characterization of orthogonality in $M_n(\bC)$, with any norm, is given in \cite[Proposition 2.1]{Li Schneider} in terms of the dual norm. Using this,  orthogonality  in $M_n(\bC)$, with induced norms, is obtained in \cite[Proposition 4.2]{Li Schneider}. In $M_n(\bC)$, with \emph{Schatten $p$-norms} ($1\leq p\leq \infty$),  characterizations of orthogonality are given in \cite[Theorem 3.2, Theorem 3.3]{Li Schneider}. For $1<p<\infty$,  this was also given  in \cite[Theorem 2.1]{Bhatia}.

Orthogonality has been characterized in more general normed spaces, namely, \emph{Hilbert $C^*$-modules}. It was shown in \cite[Theorem 2.7]{2012}  that in a Hilbert $C^*$-module $\mathscr{E}$  over a  complex unital $C^*$-algebra $\A$, an element $e_1\in \mathscr E$ is orthogonal to another element $e_2\in \mathscr E$  if and only if there exists $\phi\in S_{\A}$ such that $ \phi(\left<e_1|e_1\right>)= \|e_1\|^2 \mbox{ and } \phi(\left<e_1|e_2\right>) = 0 $. Another proof of this result was given in \cite[Theorem 4.4]{2013}.
This can  be generalized to obtain a characterization of orthogonality to subspaces of Hilbert $C^*$-modules as follows. 
\begin{theorem}\label{6}(\cite[Theorem 3.5]{2019}) Let $\mathscr{E}$ be a Hilbert $C^*$-module over a unital complex $C^*$-algebra $\A$. Let $e \in \mathscr{E}$.  Let $\mathscr{F}$ be a subspace of $\mathscr{E}$. Then $e$ is orthogonal to $\mathscr{F}$  if and only if there exists  $\phi\in S_{\A}$ such that $ \phi(\left<e|e\right>)= \|e\|^2 \mbox{ and } \phi(\left<e|f\right>) = 0$ for all $f\in \mathscr{F}$.\end{theorem}
A proof of Theorem \ref{6} can be found in \cite{2019}.
Alternatively, this can also be proved along the same lines of the proof of \cite[Theorem 2.4]{2012} by  finding a generalization of the distance formula \cite[Proposition 2.3]{2012} to a subspace. This extension of the distance formula is mentioned in the next section.

We end this section with various directions of research happening around the concept of orthogonality, where it comes into play naturally. In Hilbert $C^*$-modules, the role of scalars is played by the elements of the underlying $C^*$-algebra. Using this fact, a strong version of orthogonality was introduced in Hilbert $C^*$-modules in \cite{SBJo}. For a left Hilbert $C^*$-module, an element $e_1\in \mathscr E$ is said to be \emph{strong orthogonal} to another element $e_2\in \mathscr E$ if $\norm{e_1}\leq \norm{e_1+ae_2}$ for all $a\in\A$. Clearly, if $e_1$ is strong orthogonal to $e_2$, then we have $e_1$ is orthogonal to $e_2$. However, strong orthogonality is weaker than inner product orthogonality in $\mathscr{E}$ (see \cite[Example 2.4]{SBJo}). Necessary and sufficient conditions are studied in \cite[Theorem 3.1]{2012} and \cite[Theorem 3.5, Corollary 4.9]{SBJo 1}, when any two of these three orthogonalities coincide in a full Hilbert $C^*$-module. In \cite[Theorem 2.6]{Sym}, it was shown that in a full Hilbert $C^*$-module, strong orthogonality is symmetric if and only if Birkhoff-James  orthogonality is symmetric if and only if  strong orthogonality coincides with inner product orthogonality. Theorem 2.5 of \cite{SBJo} gives characterization of strong orthogonality in terms of Birkhoff-James orthogonality.
 
Characterizations of orthogonality are also useful in finding conditions for equality in triangle inequality in a normed space:
\begin{prop}\label{p3.8}(\cite[Proposition 4.1]{2012})
Let $V$ be a normed space. Let $ u,v\in V$. Then the following are equivalent.
\begin{enumerate}
\item $\|u+v\|=\|u\|+\|v\|$.
\item $v$ is orthogonal to $\|u\|v-\|v\|u$.
\item $u$ is orthogonal to $\|u\|v-\|v\|u$.
\end{enumerate}
\end{prop} 
This can be extended to the case of arbitrarily finite families of vectors (see \cite[Remark 4.2]{2012}). As mentioned in \cite[Remark 4.2]{2012}, a characterization of triangle equality in a pre-Hilbert $C^*$-module  given in \cite[Theorem 2.1]{trieq} can be proved using Theorem \ref{6} and  Proposition \ref{p3.8}. 
For the study of  various other equivalent conditions for equality in triangle inequality or Pythagoras equality, the interested reader is referred to \cite{trieq, 2012, trieq 2}. 

In a normed linear space $V$, an element $u$ is said to be \emph{norm-parallel} to another element $v$  (denoted by $u\parallel v$)  if $\norm{u+\lambda v} = \norm{u} + \norm{v}$ for some $\lambda\in\bF$ with $|\lambda| = 1$ \cite{Seddik}. In the case of
inner product spaces, the norm-parallel relation is exactly the usual vectorial parallel relation, that is, $u\parallel v$ if and only if $u$ and $v$ are linearly dependent.  Seddik \cite{Seddik} introduced this concept while studying elementary operators on a standard operator algebra. Interested readers for the work on elementary operators and orthogonality are referred to \cite{Anderson, Du Hong, Duggal 1,Seddik,Seddik 1,williams}, and also to \cite[Theorem 4.7]{2012} and \cite[Section 3]{Keckic 2}. 

As a direct consequence of Proposition \ref{p3.8}, norm-parallelism can be characterized using the concept of orthogonality (this characterization is also given in \cite[Theorem 2.4]{Parallel 1}). So the results on orthogonality can be used to  find results on norm-parallelism, for example, see \cite[Proposition 2.19]{Parallel 1} for $M_n(\bC)$ equipped with the Schatten $p$-norms  and \cite[Remark 2]{2017} for $(M_n(\bC),  \|\cdot\|_{(k)})$. The characterizations of norm-parallelism in $C(X)$ are given in \cite{Parallel 2}, and in  $\mathscr B_1(\hc)$ and $K(\hc)$ are given in \cite{Parallel 3}. Other  results in $\mathscr B(V_1,V_2)$ (with restrictions on $V_1$ and $V_2$, and the operators) are given in \cite{Parallel 4, Parallel 5}. Some of these results can be obtained by using \cite[Theorem 2.4]{Parallel 1} and the corresponding results on orthogonality. 
Some variants of the definition of norm-parallelism have been introduced in \cite{MalSainPaul, Parallel 3, Parallel 6}. Concepts of approximate Birkhoff-James orthogonality and $\eps$-Birkhoff orthogonality have been studied in \cite{Approx 1000, Approx 3, Approx 1, Approx 2, Approx 6,  Approx 4, Approx 5}. The idea to define these concepts of approximate Birkhoff-James orthogonality and $\eps$-Birkhoff orthogonality in a normed space is to generalize the concept of approximate orthogonality in inner product spaces, which is defined as  $v\bot^{\eps} u \iff \big|\langle v|u \rangle\big|\leq\eps\norm{v}\norm{u}$. The latter has been studied  in \cite{Approx 100} and  \cite[Section 5.2]{last}.}

\section{Distance formulas and conditional expectations}\label{Sec4}

An important connection of orthogonality with  distance formulas was noted in \eqref{95} and \eqref{96}.
From \eqref{95}, we also get that for any $A\in M_n(\bC)$, \begin{equation*}\label{99}\dist(A, \bC 1_{M_n(\bC)}) = \max\left\{\big|\bra Ax | y\ket\big| : \norm{x} = \norm{y} = 1 \text{ and } y\bot x\right\}. \end{equation*}  Using this, one obtains \begin{align}\dist(A, \bC 1_{ M_n(\bC)}) &= 2\max\{\norm{U'AU'^*-UAU^*} :  U, U' \text{ unitary}\} \nonumber\\
&= 2 \max\{\norm{AU-UA} : U \text{ unitary}\}\nonumber\\
&= 2 \max\{\norm{AT-TA} : T\in M_n(\bC), \norm{T} =1\}. \label{i}
\end{align} This was proved in \cite[Theorem 1.2]{Bhatia} and the discussion after that.  The operator $\delta_A(T) = AT-TA$ on $M_n(\bC)$ is called an \emph{inner derivation}. So this gives that  $\|\delta_A\|=2\ \dist(A, \bC 1_{M_n(\bC)})$. This was also extended to the infinite dimensional case in \cite[Remark 3.2]{Bhatia}. These results were first proved by Stampfli \cite{derivation} using a completely different approach. Since all the derivations on $\mathscr B(\hc)$ are inner derivations (see \cite[Theorem 9]{Kaplansky}), the norm of any derivation on $\mathscr B(\hc)$ is $2\ \dist(A, \bC 1_{\mathscr B(\hc)})$, for some $A\in \mathscr B(\hc)$. For $A,B\in \mathscr B(\hc)$,  let $\delta_{A, B}(T) = AT - TB$ for all $T\in\mathscr B(\hc)$.
 In \cite[Theorem 8]{derivation}, an expression for the norm of the elementary operator $\delta_{A,B}$ is given. 
In \cite[Theorem 5]{derivation}, a distance formula was obtained in an irreducible unital $C^*$-algebra as follows.

\begin{theorem}(\cite[Theorem 5]{derivation})\label{5} Let $\hc$ be a complex Hilbert space. Let $\B$ be an irreducible unital $C^*$-subalgebra of $\mathscr B(\hc)$. Let $A\in\B$. Then \begin{equation*}\label{98}2\ \dist(A, \bC 1_{\mathscr B(\hc)}) = \sup\{\norm{AT - TA} : T \in\B \text{ and } \norm{T} =1\} = \norm{\left.\delta_{A}\right|_{\B}}.\end{equation*}\end{theorem}
By the Russo-Dye theorem \cite[II.3.2.15]{Blackadar},  under the assumptions of the above theorem, we obtain\begin{equation}\label{97}2\ \dist(A, \bC 1_{\mathscr B(\hc)}) = \sup\{\norm{AU - UA} : U \in\B \text{ and } U \text{ is unitary}\}.\end{equation}

Expressions for the norm of a derivation on von Neumann algebras can be found in \cite{Gajendragadkar}. The most important fact used here is that all the derivations on von Neumann algebras are inner. This was a conjecture by  Kadison for a long time and was proved in \cite{Sakai}. More on derivations on a $C^*$-algebra can be found in \cite{Johnson, Kadison, Miles, Zsido}. A lot of work has been done to answer the question when the range of a derivation is orthogonal to its kernel. It was proved in \cite[Theorem 1.7]{Anderson} that if $N$ is a normal operator in $\mathscr B(\hc)$, then the kernel of $\delta_N$ is orthogonal to the range of $\delta_N$. In \cite[Theorem 1]{ Kittaneh 2}, it was shown that the Hilbert-Schmidt operators in the kernel of $\delta_N$ are orthogonal to the Hilbert-Schmidt operators in the range of $\delta_N$, in the usual Hilbert space sense. In \cite[Theorem 3.2(a)]{Maher},  the Schatten $p$-class operators in the kernel of $\delta_N$ were shown to be orthogonal to the  Schatten $p$-class operators in the range of $\delta_N$, in the Schatten $p$-norm.  A similar result for the orthogonality in unitarily invariant norms defined on the norm ideals of $K(\hc)$ is given in \cite[Theorem 1]{Kittaneh 3}. For related study on derivations, elementary operators and orthogonality in these normed spaces, see \cite{Duggal, Keckic 3, Kittaneh 1,  Mazouz, Mecheri 4, Mecheri 1, Mecheri 2, Mecheri 3, Eo, A10, A11}.

Similar to \eqref{97}, an expression for the distance of an element of a general $C^*$-algebra from a $C^*$-subalgebra can be obtained from the below theorem of Rieffel \cite{rieffel 1}.
\begin{theorem}\label{xxx}(\cite[Theorem 3.2]{rieffel 1}) Let $\A$  be a $C^*$-algebra. Let $\B$ be a $C^*$-subalgebra of $\A$ which contains a bounded approximate identity for $\A$. Let $a\in\A$.  Then there exists a cyclic representation $ (\hc, \pi, \xi)$ of $\A$ and a Hermitian as well as a unitary operator $U$ on $\hc$ such that $\pi(b)U = U\pi(b)$ for all $b\in\B$ and $\dist(a, \B) = \frac{1}{2} \norm{\pi(a)U - U\pi(a)}$. \end{theorem}
By Theorem \ref{xxx}, we obtain
\begin{align*}2\ \dist(a, \B) &= \max\{\norm{\pi(a)U - U\pi(a)} :  U\in\mathscr B(\hc), U = U^*, U^2 = 1_{\mathscr B(\hc)},\\  &\ \hspace{1.2cm}(\hc, \pi, \xi) \text{ is a cyclic representation of } \A,
 \text{ and }\\
 &\ \hspace{1.2cm}\pi(b)U = U\pi(b)\text{ for all } b\in\B \}.\end{align*}

Looking at the last expression  and \eqref{97}, it is tempting to conjecture that if $\A$ is a unital irreducible $C^*$-algebra, $a\in\A$ and $\B$ is a unital $C^*$-subalgebra of $\A$, then \begin{align}
2\ \dist(a, \B) &= \sup\{\norm{au - ua} : u \in\A,\ u \text{ is a unitary element, and } bu = ub\nonumber \\ & \text{ for all } b\in\B\}.\label{xxxx}\end{align}
We note that it is not possible to prove \eqref{xxxx} by proceeding along the lines of the proof of Theorem \ref{5} given in \cite{Bhatia}, which uses \eqref{i}. In particular, the following does not hold true in $M_n(\bC)$ : \begin{equation*}\dist(A, \B) = \max\left\{\big|\bra Ax | y\ket\big| : \norm{x} = \norm{y} = 1\text{ and }y\bot Bx \text{ for all } B\in\B\right\}.\label{ii}\end{equation*} 
For example, take $A = 1_{M_n(\bC)}$ 
and $\B = \{X\in M_n(\bC): \tr(X) = 0\}$. 
Then $1_{M_n(\bC)}$ is orthogonal to $\B$. Now if the above is true, then we would get unit vectors $x, y$  such that $\big|\bra x | y\ket\big| = 1$ and $\bra Bx | y\ket = 0$ for all $B\in\B$. 
Let $P= xy^*$. Then  $\rank\ P = 1$ and $\tr(BP) = 0$ for all $B\in\B$. 
 But $\tr(BP) = 0$ for all $B\in\B$ gives $P = \lambda 1_{M_n(\bC)}$ for some $\lambda\in\bC$ (see \cite[Remark 3]{2014}), which contradicts the fact that $\rank\ P = 1$. 
So this approach to prove \eqref{xxxx} does not work. However it would be interesting to know if \eqref{xxxx} is true or not. This is an open question.

The above example contradicts Theorem 5.3 of \cite{last}, which says that for Hilbert spaces $\hc$ and  $\hk$, if $A\in K(\hc, \hk)$ and  $\B$ is a finite dimensional subspace of $K(\hc, \hk)$, then
\begin{equation*}\dist(A, \B) = \sup\left\{\big|\bra Ax | y\ket\big| : \norm{x} = \norm{y} = 1\text{ and }y\bot Bx \text{ for all } B\in\B\right\}.\label{iii}\end{equation*} The proof of this theorem has a gap, after invoking Theorem 5.2, in \cite{last}.

As an application of Theorem \ref{4}, we obtain the following distance formula. 
\begin{theorem}Let $a\in\A$. Let $\B$ be a subspace of $\A$. Suppose there is a best approximation to $a$ in $\B$. Then
\begin{align}
\dist(a,\B)&=\max\left\{\big|\langle \pi(a) \xi | \eta\rangle\big|: (\hc, \pi, \xi) \text{ is a cyclic representation of } \A,\right.\nonumber\\ 
& \ \hspace{1.25cm} \eta\in\hc,  
    \|\eta\|=1 \text{ and }  \langle \pi(b)\xi|\eta\rangle=0 \text{ for all } b\in \B\big\}.\label{th4.3}\end{align}\end{theorem}

\begin{proof} Clearly $RHS\leq LHS$. To prove equality, we need to find a cyclic representation $(\hc, \pi, \xi)$ of $\A$ and a unit vector $\eta\in\hc$ such that $\dist(a, \B) = \big|\langle \pi(a) \xi | \eta\rangle\big|$ and $\langle\pi(b)\xi | \eta\rangle=0$ for all $b\in\B$. Let $b_0$ be a best approximation to $a$ in $\B$. By Theorem \ref{4},  there exists $\phi\in S_{\A}$ such that $$\phi((a-b_0)^*(a-b_0)) = \norm{a-b_0}^2$$ and $$\phi((a-b_0)^*b)=0 \text{ for all }b\in \B.$$ Now there exists a cyclic representation $(\hc, \pi, \xi)$ such that $\phi(c) = \bra \pi(c)\xi|\xi\ket$ for all $c\in\A$. So $\norm{\pi(a-b_0)\xi} = \norm{a-b_0}$ and $\bra\pi(a-b_0)\xi | \pi(b)\xi\ket = 0$ for all $b\in\B$. Taking $\eta = \frac{1}{\norm{a-b_0}}\pi(a-b_0)\xi$, we get the required result.\end{proof}

The authors have recently observed in \cite{2019} that the above theorem also holds true without the existence of a best approximation to $a$ in $\B$. Notice that the right hand side of \eqref{th4.3} uses only algebraic structure of $\A$ (as  cyclic representations are defined by the algebraic structure of $\A$). More such distance formulas using only the algebraic structure of $\A$ are also known. When $\B=\mathbb C1_\A$, Williams \cite[Theorem 2]{williams} proved that for $a\in \A$, \begin{equation}\dist(a, \mathbb C1_\A)^2 =\max\{\phi(a^*a)-|\phi(a)|^2:\phi \in  S_{\A}\}.\label{rie}\end{equation} When $\A=M_n(\bC)$, another proof of \eqref{rie} was given by Audenaert \cite[Theorem 9]{audenaert}. Rieffel \cite[Theorem 3.10]{rieffel} obtained \eqref{rie}, using a different method. In \cite{rieffel}, it was also desired to have a generalization of \eqref{rie} with $\mathbb C1_\A$ replaced by a unital $C^*$-subalgebra.  For $\A=M_n(\bC)$, a formula in this direction was obtained in \cite[Theorem 2]{2014}. 
An immediate application of Theorem \ref{4} gives the following generalization of \eqref{rie}, when $\mathbb C1_\A$ is replaced by a subspace $\B$ of $\A$ and there is a best approximation to $a$ in $\B$. 
\begin{theorem}\label{4.4}(\cite[Corollary 1.2]{2019}) Let $a\in\A$. Let $\B$ be a subspace of $\A$. Let $b_0$ be a best approximation to $a$ in $\B$. Then
\begin{equation*}
\dist(a,\B)^2 = \max\{\phi(a^*a) -\phi(b_0^*b_0) : \phi\in S_{\A} \text{ and }\phi(a^*b) = \phi(b_0^*b) \text{ for all } b\in\B\}.\label{xx}
\end{equation*} \end{theorem}
For details, see \cite{2019}. Geometric interpretations of Theorem \ref{4} and Theorem \ref{4.4} have also been explained in \cite{2019}. 

Henceforward,  $C^*$-algebras are assumed to be complex $C^*$-algebras. Another distance formula, which is a generalization of \cite[Proposition 2.3]{2012}, is given below. Some notations are in order. Given $\phi\in S_{\A}$, let $\sL = \{c\in\A : \phi(c^*c) = 0\}$, and let $\bra a_1+\sL|a_2+\sL\ket_{\A/\sL} = \phi(a_1^*a_2)$, for all $a_1, a_2\in \A$. Then $\A/\sL$ is an inner product space. For $a\in \A$, let $b_0$ be a best approximation to $a$ in $\B$. Let $$M_{a, \B}(\phi) = \sup\{\phi((a-b_0)^*(a-b_0)) - \sum_{\alpha}\lvert\phi((a-b_0)^*b_\alpha)\rvert^2\},$$ where the supremum is taken over all orthonormal bases $\{b_\alpha+ \sL\}$  of  $\B/\sL $ in $\A/\sL$.

\begin{theorem}\label{7} Let $a\in\A$. Let $\B$ be a subspace of $\A$. Let $b_0$ be a best approximation to $a$ in $\B$.  Then $$\dist(a, \B)^2= \max \{M_{a, \B}(\phi): \phi\in S_{\A} \}.$$\end{theorem}
\begin{proof}
Clearly $RHS\leq LHS$. For an orthonormal basis $\{b_\alpha+ \sL\}$  of  $\B/\sL$, we have $$\phi((a-b_0)^*(a-b_0)) - \sum_{\alpha}\lvert\phi((a-b_0)^*b_\alpha)\rvert^2 \leq LHS.$$ And equality occurs because by Theorem \ref{4}, there exists $\phi\in S_{\A}$ such that $\phi((a-b_0)^*(a-b_0)) = \dist(a, \B)^2$ and $\phi((a-b_0)^*b) = 0$ for all $b\in\B$.
\end{proof}
Now along the lines of the proof of \cite[Theorem 2.4]{2012} and using Theorem \ref{7}, we get the next result. For a Hilbert $C^*$-module $\mathscr E$ over $\A$ and $\phi\in S_{\A}$, let $\sL = \{{e}\in \mathscr E: \phi(\bra {e} | {e}\ket) = 0\}$. On $\mathscr{E}/\sL$, define an inner product as $\bra {e_1}+\sL|{e_2}+\sL\ket_{\mathscr{E}/\sL} = \phi(\bra {e_1}|{e_2}\ket)$ for all ${e_1}, {e_2}\in   \mathscr E$. For $e\in \mathscr E$, let $f_0$ be a best approximation to $e$ in $\mathscr{F}$. Let $$M_{e, \mathscr{F}}(\phi) = \sup\{\phi(\bra e-f_0 | e-f_0\ket) - \sum_{\alpha}\lvert\phi(\bra e-f_0 | f_\alpha\ket)\rvert^2\},$$ where the supremum is taken over all orthonormal bases  $\{f_\alpha+ \sL\}$ of  $\mathscr{F}/\sL$  in $\mathscr{E}/\sL$.
\begin{theorem}\label{7777} Let $\mathscr{E}$ be a Hilbert $C^*$-module over $\A$. Let $e\in \mathscr{E}$. Let $\mathscr{F}$ be a subspace of $\mathscr{E}$.  Let $f_0$ be a best approximation to $e$ in $\mathscr{F}$. Then  $$\dist(e, \mathscr{F})= \max\{M_{e, \mathscr{F}}(\phi): \phi\in S_{\A}\}.$$ 
\end{theorem}


Rieffel \cite[p. 46]{rieffel} had questioned to have expressions of distance formulas  in terms of \emph{conditional expectations}. We end the discussion on distance formulas with our progress in this direction.  For a $C^*$-algebra $\A$ and a $C^*$-subalgebra $\B$ of $\A$,  a {conditional expectation} from $\A$ to $\B$  is a completely positive map $E: \A\rightarrow \B$ of unit norm such that  $E(b) = b$, $E(b a) = bE(a)$ and $E(a b) = E(a)b$, for all $a\in\A$ and $b\in\B$  \cite[p. 141]{Blackadar}. In fact  any projection $E:\A\rightarrow \B$ of norm one is a conditional expectation and vice-a-versa (see \cite[Theorem II.6.10.2]{Blackadar}). An interesting fact is that a map $E:\A\rightarrow \B$ is a conditional expectation if and only if $E$ is idempotent, positive and satisfies $E(b_1 a b_2) = b_1E(a)b_2$, for all $a\in\A$ and $b_1, b_2\in\B$  (see \cite[Theorem II.6.10.3]{Blackadar}). Thus conditional expectations from $\A$ to $\B$ are also determined completely by the algebraic structure. A  Banach space $V_1$ is said to be \emph{injective} if for any inclusion of Banach spaces $V_3 \subseteq  V_2$, every bounded linear mapping $f_0:  V_3 \rightarrow V_1$ has a linear extension $f :  V_2 \rightarrow  V_1$ with $\norm{f} = \norm{f_0}$. A Banach space  is injective if and only if it is isometrically isomorphic to  $C(X)$, where $X$ is a compact Hausdorff space in which closure of any open set is  an open set  (see \cite[p. 70]{ruan}).  For $v\in  V$ and $W$ a subspace of $ V$, let $\langle v, W\rangle$ denote the subspace generated by $v$ and $W$.

\begin{theorem}\label{21} Let $a\in\A$. Let $\B$ be a subspace of $\A$ such that $\B$ is an injective Banach space and $1_\A\in \B$. Suppose there is a best approximation to $a$ in $\B$. Then there exists  $\phi\in S_{\A}$ and a projection $E: \A \rightarrow \B$ of norm atmost two such that $\phi\circ E = \phi$ and $\dist(a,\B)^2 = \phi(a^*a) -\phi(E(a)^*E(a))$.\end{theorem}

\begin{proof} Let $b_0$ be a best approximation to $a$ in $\B$. We define $\tilde{E}: \langle a, \B\rangle \rightarrow \B$ as $\tilde{E}(b) =b$ for all $b\in\B$ and $\tilde{E}(a) = b_0$ and extend it linearly on  $\langle a, \B\rangle$. Using Theorem \ref{4}, there exists $\phi\in S_{\A}$ such that $\dist(a, \B)^2 = \phi(a^*a) - \phi(b_0^*b_0)$ and $\phi(a^*b) = \phi(b_0^*b)$ for all $b\in\B$. Since $1_{\A}\in \B$, we get $\phi(a) = \phi(b_0) = \phi(E(a))$. And clearly $\phi(b) = \phi(\tilde{E}(b))$ for all $b\in \B$. Thus $\phi\circ \tilde{E} = \phi$. 
 Since $b_0$ is a best approximation to $a$ in  $\B$, $\|a-b_0\|\leq \|a\|$.  So $\norm{b_0}\leq2\norm{a}$. Now  let $b\in \B$ and $\alpha\in \bC$. Then $\tilde{E}(\alpha a+b) = \alpha b_0+b$ and $\alpha b_0+b$ is a best approximation to $\alpha a+b$ in $\B$. Thus $\norm{\alpha b_0+b}\leq2\norm{\alpha a+b}$. Hence $\|\tilde{E}\| \leq 2$. Since $\B$ is injective,  there exists a linear extension $E: \A \rightarrow \B$ with norm same as that of $\tilde{E}$. This $E$ is the required projection.\end{proof} 

For any given conditional expectation $E$ from $\A$ to $\B$, we can define a $\B$-valued inner product on $\A$ given by $\bra a_1|a_2\ket_{E} = E(a_1^*a_2)$ (see \cite{rieffel}). In \cite{2019}, we obtain a lower bound for $\dist(a, \B)$ as follows. \begin{theorem}\label{211} Let $a\in\A$. Let $\B$ be a $C^*$-subalgebra of $\A$ such that $1_\A\in \B$.  Then
\begin{align}\dist(a, \B)^2&\geq  \sup\{\phi(\bra a-E(a)|a-E(a)\ket_{E}): \phi\in\mathcal S_{\A}, E \text{ is a conditional} \nonumber\\
& \text{ expectation from } \A \text{ onto } \B \}.\label{10000}\end{align}
(Here we follow the convention that $\sup(\emptyset) = -\infty$.)\end{theorem}

\textbf{Remark:} It is worth mentioning here that if in Theorem \ref{21}, we take $\B$ to be a $C^*$-algebra and we are able to find a projection of norm one, then  we will get equality in \eqref{10000}, that is,  \begin{align*}\dist(a, \B)^2 & =\sup\{\phi(\bra a-E(a)|a-E(a)\ket_{E}): \phi\in\mathcal S_{\A}, E \text{ is a conditional} \nonumber\\
& \text{ expectation from } \A \text{ onto } \B \}.\end{align*}
This happens in the special case when $\B = \bC1_{\A}$, because for any $c\in \A$, the norm of the best approximation of $c$ to $\bC 1_{\A}$ is less than or equal to $\|c\|$, and thus the norm of the projection $E$ in Theorem \ref{21} is one.

\subsection*{Acknowledgements}

The authors would like to thank Amber Habib and Ved Prakash Gupta for many useful discussions. The authors would also like to acknowledge very helpful comments by the referees.

The research of P. Grover is supported by INSPIRE Faculty Award IFA14-MA-52 of DST, India, and by Early Career Research Award ECR/2018/001784 of SERB, India.

\end{document}